\documentclass[
 aip,
 amsmath,amssymb,
 reprint,floatfix]{revtex4-1}
\usepackage{graphicx}
\usepackage{dcolumn}
\usepackage[mathcal]{euscript}
\usepackage{bm}
\usepackage[utf8]{inputenc}
\usepackage[T1]{fontenc}
\usepackage{mathptmx}
\usepackage{etoolbox}
\usepackage{amsfonts}
\usepackage{amssymb}
\usepackage{amsthm}
\usepackage{amssymb}
\usepackage{float}
\usepackage[all]{xy}
\usepackage{graphicx}
\usepackage{changes}
\usepackage{hyperref}
\usepackage{bbm}
\usepackage{xcolor}
\usepackage[caption=false]{subfig}
\usepackage{subcaption} 
\usepackage{cancel} 

\newcommand{\T}{\mathcal T}

\newcommand{\xb}{\boldsymbol x}
\newcommand{\yb}{\boldsymbol y}
\newcommand{\Fc}{\mathcal F}
\newcommand{\Kc}{\mathcal K}
\newcommand{\gb}{\boldsymbol g}

\newcommand{\psib}{\boldsymbol \psi}
\newcommand{\Psib}{\boldsymbol \Psi}
\newcommand{\Cb}{\boldsymbol C}

\newcommand{\Kb}{\boldsymbol K}
\newcommand{\xib}{\boldsymbol \xi}
\newcommand{\wb}{\boldsymbol w}
\newcommand{\vb}{\boldsymbol v}
\newcommand{\Bb}{\boldsymbol B}
\newcommand{\Mb}{\boldsymbol M}
\newcommand{\Jb}{\boldsymbol J}
\newcommand{\Ib}{\boldsymbol I}
\newcommand{\Ab}{\boldsymbol A}
\newcommand{\Xib}{\boldsymbol \Xi}
\newcommand{\Lambdab}{\boldsymbol \Lambda}
\newcommand{\Wb}{\boldsymbol W}
\newcommand{\Db}{\boldsymbol D}
\newcommand{\ub}{\boldsymbol u}
\newcommand{\omegab}{\boldsymbol \omega}
\newcommand{\Ob}{\boldsymbol 0}
\newcommand{\deltab}{\boldsymbol \delta}

\theoremstyle{plain}
\newtheorem{theorem}{Theorem}[section]

\newtheorem{algorithm}[theorem]{Algorithm}

\newtheorem{proposition}[theorem]{Proposition}
\newtheorem{remark}[theorem]{Remark}

\makeatletter
\def\@email#1#2{%
 \endgroup
 \patchcmd{\titleblock@produce}
  {\frontmatter@RRAPformat}
  {\frontmatter@RRAPformat{\produce@RRAP{*#1\href{mailto:#2}{#2}}}\frontmatter@RRAPformat}
  {}{}
}%
\makeatother
\begin{document}

\preprint{AIP/123-QED}

\title[Featurizing Koopman Mode Decomposition]{Featurizing Koopman Mode Decomposition for Robust Forecasting}
\author{David Aristoff}
\affiliation{Colorado State University, Fort Collins, CO, 80523, USA}
\thanks{Author to whom correspondence should be addressed: \url{aristoff@colostate.edu}}
\author{Jeremy Copperman}
\affiliation{Oregon Health \& Science University, Cancer Early Detection Advanced Research Center, Knight Cancer Institute, Portland, OR, 97201, USA}
\author{Nathan Mankovich}
 \affiliation{University of Valencia, Val{\`e}ncia, 46010, Spain}
\author{Alexander Davies}
\affiliation{Oregon Health \& Science University, Cancer Early Detection Advanced Research Center, Knight Cancer Institute, Portland, OR, 97201, USA}
\affiliation{Oregon Health \& Science University, Division of Oncological Science, Knight Cancer Institute, Portland, OR, 97201, USA}

\date{\today}

\begin{abstract}
This article introduces an advanced Koopman mode decomposition (KMD) technique -- coined {\em Featurized Koopman Mode Decomposition} (FKMD) -- that uses delay embedding and a learned Mahalanobis distance to enhance analysis and prediction of high dimensional dynamical systems. The delay embedding expands the observation space to better capture underlying manifold structure, while the Mahalanobis distance adjusts observations based on the system's dynamics. This aids in {\em featurizing} KMD in cases where good features are not a priori known. We show that FKMD improves predictions for a high-dimensional linear oscillator, a high-dimensional Lorenz attractor that is partially observed, and a cell signaling problem from cancer research. 
\end{abstract}

\maketitle

\section{Introduction}






Koopman mode decomposition~\cite{mezic2005spectral,mezic2021koopman} (KMD) has emerged as a powerful tool for analyzing nonlinear dynamical systems. The power of KMD comes from lifting the nonlinear dynamics into a vector space of feature functions; 
the evolution on this space is described by the linear {\em Koopman operator}\cite{koopman1931hamiltonian,koopman2004modeling}. Through this trick, KMD can identify patterns and coherent structures that evolve linearly in time. 

KMD enables both quantitative predictions and qualitative analysis of dynamics~\cite{williams2015data,tu2013dynamic}. The framework for nonlinear features was introduced by Williams {et al}~\cite{williams2015data}. Since then, KMD has been kernelized~\cite{kevrekidis2016kernel}, integrated with control theory~\cite{proctor2016dynamic}, sped up with random Fourier features~\cite{degennaro2019scalable}, used with time delay embeddings~\cite{kamb2020time}, viewed from the perspective of Gaussian processes~\cite{kawashima2022gaussian}, and imposed with physical constraints~\cite{baddoo2023physics}. KMD has  been widely applied, including in infectious disease control~\cite{koopman2004modeling}, video~\cite{erichson2019compressed}, neuroscience~\cite{brunton2016extracting}, fluid dynamics~\cite{mezic2013analysis,bagheri2013koopman,arbabi2017study}, molecular dynamics~\cite{wu2017variational,klus2020data}, and climate science~\cite{navarra2021estimation}. For recent advances, challenges, and open problems in data-driven Koopman learning, see~\cite{brunton2021modern,bevanda2021koopman,colbrook2023multiverse}.

Kernel KMD, which uses kernel features, is a natural choice when system-specific feature functions are unknown~\cite{kevrekidis2016kernel}. The choice of kernel can have a large effect on the quality of KMD. The most commonly used kernels are isotropic Gaussian or Mat{\'e}rn kernels~\cite{genton2001classes}, which give uninformative measures of distance in high dimension. 
Artificial neural networks are natural competitors to KMD that can overcome this curse of dimensionality, but they cannot identify linearly evolving 
structures and require tuning over many hyperparameters.


We propose a novel method called {\em Featurized Koopman Mode Decomposition} (FKMD). 
Our method featurizes KMD by learning a Mahalanobis distance-based kernel~\cite{radhakrishnan2022feature}. 
This kernel prioritizes  dynamically important directions by enforcing isotropic changes in space and time (see Theorem~\ref{thm:main}). 
This mitigates the curse of dimensionality,  
leading to improvements over ordinary Gaussian kernel KMD. 


FKMD includes three key ingredients: (i) kernels that use a learned Mahalanobis distance; (ii) nonstandard delay embeddings; and (iii)  efficient implementation with random Fourier features. Delay embeddings, which extend data arrays by including time history, allow reconstruction of underlying manifolds~\cite{kamb2020time,takens2006detecting}. 
While delay embedding of features has been introduced in~\cite{arbabi2017ergodic}, 
we apply an additional  embedding to the samples. The Mahalanobis distance finds 
appropriate time correlation structure between these delay-embedded samples, while random Fourier features enable fast computations~\cite{degennaro2019scalable}.

In sum, the \textbf{contributions} of this work are: 
\begin{itemize}
    \item {\em We introduce a new method, FKMD,} that learns features of high dimensional, delay-embedded data by encoding them in a Mahalanobis distance, leading to more effective KMD analysis and inference. We 
    show how to integrate our method with random Fourier features~\cite{rahimi2007random,nuske2023efficient}  to handle large datasets.

    \item {\em We illustrate the power of FKMD} in  three experiments. The first experiment shows how FKMD improves forecasting on a large system of linear differential equations with effective low dimensionality. In the second experiment, FKMD accurately predicts evolution of a high-dimensional Lorenz attractor~\cite{lorenz1996predictability} despite training data that is low-dimensional and noisy. Our last experiment applies FKMD to cancer cell imaging, predicting cell-signaling patterns hours into the future. 
\end{itemize}

\begin{table}[ht]
    \caption{Definitions of symbols used in this work.}
    \label{tab:symbol-definitions}
    \centering
    \begin{tabular}{c | @{\hspace{1em}}l@{}}
        Symbol & Definition \\ [0.5ex] \hline
                $\xb(t)$ & system state at time $t$ \\
        $\xb$, $\xb'$ & states, or sample points \\
        $\gb(\xb)$ & $1 \times L$ real observation function \\
        $\nabla \gb(\xb)$ & {\em transpose} of Jacobian matrix; has $L$ columns \\
              $\tau$ & evolution time step, or lag\\
        $\Fc_\tau(\xb)$ & evolution map at lag $\tau$ \\
        $\Kc_\tau(\gb)$ & Koopman operator at lag $\tau$\\
         $N$ & number of samples \\
         $R$ & number of features \\
        $\xb_1,\ldots,\xb_N$ & input samples \\
        $\yb_1,\ldots,\yb_N$ & output samples; $\yb_n = \Fc_\tau(\xb_n)$  \\
        $\psi_1(\xb),\ldots,\psi_R(\xb)$ & scalar-valued feature functions\\
        $\psib = \begin{bmatrix}\psi_1 & \ldots & \psi_R\end{bmatrix}$ & $1 \times R$ vector of feature functions \\
        $\Psib_{\xb}$ & $N \times R$ input samples $\times$ features matrix  \\
        $\Psib_{\yb}$ & $N \times R$ output samples $\times$ features matrix  \\
       $\Kb$ & $R \times R$ Koopman matrix in feature space  \\ 
         $\Bb$ & $R \times L$ observation matrix in feature space\\ 
       $\phi_m(\xb)$ & scalar-valued Koopman eigenfunctions \\
       $\vb_m^\ast$ & $1 \times L$ Koopman modes \\
       $\mu_m$ & Koopman eigenvalues \\
       $\lambda_m = \tau^{-1} \log \mu_m$ & continuous-time Koopman eigenvalues \\
       $k_{\Mb}(\xb,\xb')$ & kernel function \\
       $\Mb^{1/2}$ & change of variables matrix \\
       $\Ab$ & matrix defining a system of linear diff eqs \\
       $\Ib$ & identity matrix \\
       $\tilde{\xb}_n$, $\tilde{\yb}_n$, $\tilde{\gb}$, $\tilde{\Fc}_\tau$ & $\xb_n$, $\yb_n$, $\gb$, and $\Fc_\tau$ in the changed variables
    \end{tabular}
\end{table}

\section*{Overview of Koopman mode decomposition}

We consider 
a dynamical 
system in real Euclidean space, with evolution map $\Fc_\tau$. Given the current 
state, $\xb(t)$, the 
state at time $\tau$ into the future 
is $\Fc_\tau(\xb(t))$. That is, 
\begin{equation}\label{eq:forward_operator}
     \xb(t+\tau) = \Fc_\tau (\xb(t)).
\end{equation}

In realistic application problems, $\Fc_\tau$ is typically a complicated nonlinear 
function, {\em e.g.} a stochastic or ordinary differential equation time step, or simply a black box mapping from which some data is measured. 

The 
{\em Koopman operator} provides a dual interpretation of evolution which is {\em linear} (see Figure~\ref{fig:koopman_concept}). Namely, for an observation function $\gb(\xb)$ on the system states, the {\em Koopman operator}~\cite{mezic2013analysis,brunton2021modern,mauroy2020koopman} determines the observations at time $\tau$ in the future:  
\begin{equation}\label{eq:Koopman-definition}
 \Kc_\tau(\gb)(\xb) := \gb(\Fc_\tau(\xb)).
\end{equation}
 
 Here, $\Kc_\tau$ is a linear operator that exactly describes the evolution map $\Fc_\tau$.  
 In principle, $\gb$ could be any measurement or quantity of interest. In our experiments below, we define observations using either the full state or  delay embeddings of low-dimensional observations.

While the linear framework does not remove the complexity inherent in $\Fc_\tau$, it provides a starting point for globally linear techniques: we can apply linear analysis in~\eqref{eq:Koopman-definition} without resorting to 
local linearization of~\eqref{eq:forward_operator}. From this point of view, we can construct finite-dimensional approximations of $\Kc_\tau$ by choosing a collection of {\em feature functions} ~\cite{bishop2006pattern}, denoted $\psi_r$, that are 
evaluated at {\em sample points}. 

To this end, we choose scalar-valued features $$\psib(\xb) = \begin{bmatrix}\psi_1(\xb) & \ldots & \psi_R(\xb)\end{bmatrix},$$ 
and obtain a set of input and output sample points $\xb_1,\ldots,\xb_N$ and $\yb_1,\ldots,\yb_N$, where $\yb_n = \Fc_\tau(\xb_n)$. From these we form $N \times R$ matrices $\Psib_{\xb}$ and $\Psib_{\yb}$ whose rows are samples and columns are features:  
\begin{equation}\label{eq:Psix}
    \Psib_{\xb} = \begin{bmatrix} \psi_1(\xb_1) & \ldots &\psi_R(\xb_1) \\ 
    \vdots &  & \vdots \\
    \psi_1(\xb_N) & \ldots &\psi_R(\xb_N)
    \end{bmatrix}
\end{equation}
and 
\begin{equation}\label{eq:Psiy}
    \Psib_{\yb} = \begin{bmatrix} \psi_1(\yb_1) & \ldots &\psi_R(\yb_1) \\ 
    \vdots &   &\vdots \\
    \psi_1(\yb_N) & \ldots& \psi_R(\yb_N)
    \end{bmatrix}.
\end{equation}

 A finite-dimensional approximation, $\Kb$, of the Koopman operator should, up to estimation errors, satisfy
 \begin{equation}\label{eq:Koopman-solve}
     \Psib_{\xb} \Kb = \Psib_{\yb}.
 \end{equation}
 Here $\Kb$ is a $R \times R$ matrix, and this is a linear system that can be solved with standard methods like ridge regression. We think of $\Kb$ as acting in the {\em feature space}.

The non-linear evolution of system states  (equation~\eqref{eq:forward_operator}) and the corresponding linear Koopman operator on feature functions (equation~\eqref{eq:Koopman-definition}) are summarized in Figure~\ref{fig:koopman_concept}, which was inspired by Williams et al.~\cite{williams2015data}.

\begin{figure}[ht!]
    \centering
    \includegraphics[width = \linewidth]{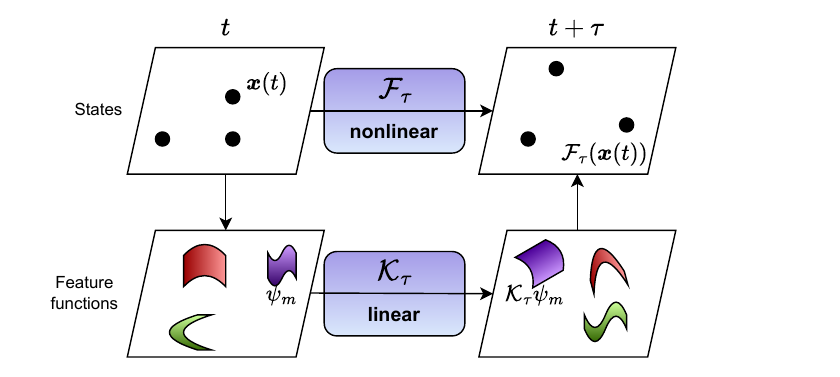}
    \caption{Illustration of the evolution map and Koopman operator. The evolution map, $\Fc_\tau$, and the Koopman operator $\mathcal{K}_\tau$, respectively evolve states ({\em e.g.}, $\xb(t)$) and features ({\em e.g.}, $\psi_m$) $\tau$ time steps in the future. The evolution of features (bottom row) is {\em linear}.}
    \label{fig:koopman_concept}
\end{figure}

If $\gb$ is a $1 \times L$ vector-valued function, we also express 
 $\gb$ in feature space coordinates as a $R \times L$ matrix $\Bb$:
 \begin{equation}\label{eq:B-solve}
   \Psib_{\xb} \Bb =     \begin{bmatrix} \gb(\xb_1) \\ \vdots \\ \gb(\xb_N) \end{bmatrix}.
 \end{equation}
Note that~\eqref{eq:B-solve} can be solved in the same manner as~\eqref{eq:Koopman-solve}.
 
 Koopman mode decomposition
 converts an eigendecomposition of $\Kb$
 back to the {\em sample space},  to interpret and/or predict 
 the dynamics defined by $\Fc_\tau$. 
 To this end, write the eigendecomposition of $\Kb$ as 
 \begin{equation}\label{eq:Koopman-decomposition}
     \Kb = \sum_{m=1}^R \mu_m \xib_m \wb_m^\ast,
 \end{equation}
 where $\mu_m$ are the eigenvalues of $\Kb$, and $\xib_m$, $\wb_m$ are the right and left eigenvectors, respectively, 
 scaled so that $\wb_m^\ast \xib_m = 1$. That is, $\Kb \xib_m = \mu_m \xib_m$ and $\wb_m^\ast \Kb = \mu_m \wb_m^\ast$.
By converting this to sample space, it can be shown that\cite{williams2015data} 
 \begin{equation}\label{eq:estimate-observable}
     \Kc_\tau(\gb)(\xb) \approx \sum_{m=1}^R e^{\tau\lambda_m} \phi_m(\xb) \vb_m^\ast,
 \end{equation}
 with $e^{\tau \lambda_m} = \mu_m$ the {\em Koopman eigenvalues},  $\phi_m(\xb) = \psib(\xb)\xib_m$  the {\em Koopman eigenfunctions}, and $\vb_m^\ast = \wb_m^\ast \Bb$ the {\em Koopman modes}. The right-hand side of~\eqref{eq:estimate-observable} is a finite-dimensional approximation of the Koopman operator at lag $\tau$ evaluated on $\gb$. 
 See  Appendix~\ref{appendix:inference} for a derivation of equation~\eqref{eq:estimate-observable}.
 
 With the Koopman eigenvalues, Koopman eigenfunctions, and Koopman modes in hand, equation~\eqref{eq:estimate-observable} can be used to predict observations 
 of the system at future times, as well 
 as analyze qualitative behavior. 
 There has been much work in this direction; we will not give a complete review, but refer to~\cite{williams2015data,arbabi2017ergodic} for the basic ideas and to~\cite{navarra2021estimation,baddoo2022kernel,nuske2023finite,wu2017variational,klus2020data,klus2020kernel} for recent applications and extensions.
 Of course, the quality of the approximation in~\eqref{eq:estimate-observable} is sensitive to the 
 choice of features and sample 
 space. With enough features 
 and samples, actual equality in~\eqref{eq:estimate-observable} 
 can be approached~\cite{williams2015data,arbabi2017ergodic}. 
 In realistic applications, 
 samples and features are limited 
 by computational constraints.

 \section{Methods}

 \subsection{Overview}

 We make two data-driven choices that can give 
 remarkably good results on complex systems. These choices are:
 \begin{itemize}

     \item[(i)] We learn a change of variables $\xb \to \Mb^{1/2} \xb$ to help define features. The matrix $\Mb$ is updated iteratively and reflects the 
     underlying system's dynamics.

     \item[(ii)] We use a ``double'' delay embedding to construct feature space, where both the sample points $\xb_n$ and the features $\psi_m$ are delay embedded.

 \end{itemize}

  Our features are based on kernels~\cite{navarra2021estimation,klus2020kernel}. Kernel features are a common choice when system-specific feature functions are not {\em a priori} known. 
 The kernels are centered around the sample points, \begin{equation}\label{eq:features}
         \psi_m(\xb) = k_{\Mb}(\xb,\xb_m), \qquad m=1,\ldots,R.
     \end{equation}
     Here, $R=N$ and 
     $k_{\Mb}$ is the kernel function
\begin{equation}\label{eq:kernel}
         k_{\Mb}(\xb,\xb') = \exp\left[-\frac{1}{2}|\Mb^{1/2} (\xb-\xb')|^2\right]
     \end{equation}
     defined using the Mahalanobis distance, $|\Mb^{1/2}(\xb-\xb')|^2 = {(\xb-\xb')^*\Mb(\xb-\xb')}$, between pairs $\xb,\xb'$. 
     In Section~\ref{sec:scaling_up}, we show how to scale up to large sample size $N> 10^5$ by using $R \ll N$ random Fourier features that estimate these kernels.

 \subsection{Change of variables}\label{sec:transformation}

Inspired by the recent work~\cite{radhakrishnan2022feature} on understanding neural networks and improving kernel methods, we target a matrix $\Mb$ using a gradient outer product structure~\cite{li1991sliced,trivedi2014consistent,radhakrishnan2022feature}. Up to a scalar bandwidth factor, we use
     \begin{equation}\label{eq:grad-outerproduct}
         \Mb = \frac{1}{N}\sum_{n=1}^N \Jb(\xb_n) \Jb(\xb_n)^\ast,
     \end{equation}
   with the ideal $\Jb$ given by
     \begin{equation}\label{eq:Jexact}
\Jb(\xb) = \lim_{\tau \to 0} \tau^{-1} [\nabla (\gb \circ \Fc_\tau)(\xb)- \nabla \gb(\xb)].
\end{equation}
Here, the gradient forms a {column} vector for each scalar observation: $\nabla \gb$ has $L$ columns, each one a gradient. In practice, we use finite $\tau$ matching the lag time between sample points (equation~\ref{eq:curvature}).

      Compared with a standard Gaussian kernel, equation~\eqref{eq:kernel}
     comes from the change of variables $\xb \mapsto \tilde{\xb} = \Mb^{1/2}\xb$. Assuming $\Mb$ is invertible, let
     $\tilde{\gb}(\tilde{\xb}) = \gb(\xb)$ and $\tilde{\Fc}_\tau(\tilde{\xb}) = \tilde{\yb}$, where $\yb = {\Fc}_\tau(\xb)$, and define 
    \begin{equation}\label{eq:tildeJ}
\tilde{\Jb}(\xb) = \lim_{\tau \to 0} \tau^{-1} [\nabla (\tilde{\gb
} \circ \tilde{\Fc}_\tau)(\xb)- \nabla \tilde{\gb}(\xb)].
\end{equation}

 This change of variables, {\em i.e.,} $\xb \mapsto \tilde{\xb}$, $\gb \mapsto \tilde{\gb}$, and $\Fc_\tau \mapsto \tilde{\Fc}_\tau$, assumes that the input/output pairs are transformed by $\Mb^{1/2}$, but that the observations do not change. 
  The matrices $\Jb$ and $\tilde{\Jb}$ measure   changes in space and time of the original and transformed variables, respectively. Theorem~\ref{thm:main} below shows that these changes are isotropic in the transformed variables (proof is in Appendix~\ref{appendix:Mahalanobis}). 

\begin{theorem}\label{thm:main}
If $\Mb$ defined by~\eqref{eq:grad-outerproduct}-\eqref{eq:Jexact} is invertible, then
\begin{equation*}
   \frac{1}{N}\sum_{n=1}^N \left|\ub^\ast \tilde{\Jb}(\tilde{\xb}_n)\right|^2 \equiv 1, \quad \textup{for all unit $\ub$}.
   \end{equation*}
   If in addition $\frac{d}{dt}\xb(t)^\ast = \xb(t)^\ast \Ab$ and $\gb(\xb) = \xb^\ast$, then $\Mb = \Ab \Ab^\ast$,  $\Jb = \Ab$, and $\tilde{\Jb} = \Mb^{-1/2} \Ab$ is an orthogonal matrix.
    \end{theorem}

Intuitively, $\Jb$ and $\tilde{\Jb}$ measure changes in space and time of $\xb$ and $\tilde{\xb}$ respectively. Theorem~\ref{thm:main} states that the change of variables, using $\xb \mapsto \tilde{\xb}$, results in isotropic changes in space and time. This property is summarized in Figure~\ref{fig:isotropic_concept}.

\begin{figure}[!ht]
    \centering
    \includegraphics[width = \linewidth]{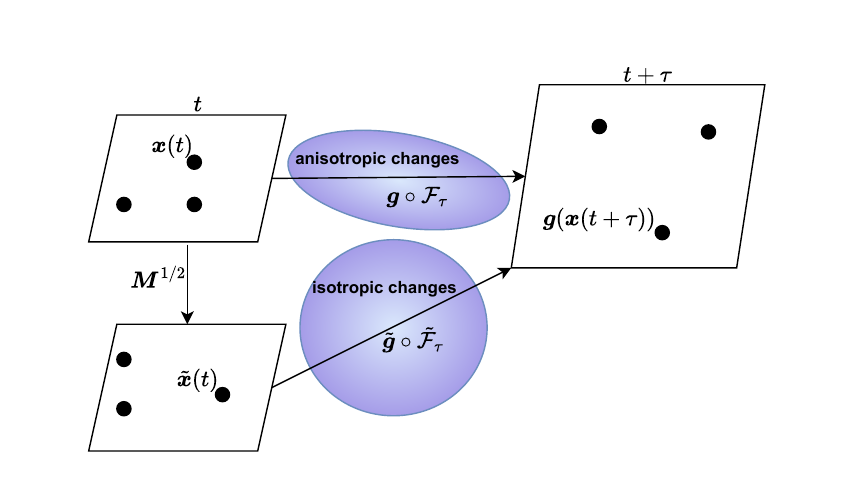}
    \caption{Illustration of Theorem~\ref{thm:main}: Changing variables from $\xb(t)$ to $\tilde{\xb}(t)$ (by multiplying by $\Mb^{1/2}$) enforces isotropic changes in observation function $\gb$ over a time interval $\tau$.}
    \label{fig:isotropic_concept}
\end{figure}

In practice, we compute $\Jb$ using 
\begin{equation}\label{eq:curvature}
        \Jb(\xb) \approx \tau^{-1} \sum_{m \in {\mathcal S}} (\mu_m-1) \nabla \phi_m(\xb) \vb_m^\ast;
     \end{equation}
     this is an estimate of~\eqref{eq:Jexact} based on equation~\eqref{eq:estimate-observable}, finite lag $\tau$, and mode set ${\mathcal S}$.
     The set ${\mathcal S}$ of modes is chosen according to cutoffs intended to enforce stability and eliminate noise effects (see Section~\ref{sec:tuning}), and can be found using cross-validation. 
     
     In practice, $\Mb$ may be noninvertible or approximately noninvertible, because it maps away variables that are irrelevant to the observation $\gb$, along with noise identified by mode selection. From this perspective, the $\Mb^{1/2}$ transformation removes unnecessary variables, and enforces isotropy of the remaining variables in the sense of Theorem~\ref{thm:main}.

   
\subsection{Delay embeddings}
     
 The other key ingredient in FKMD involves delay embeddings. According to Taken's theorem\cite{takens2006detecting}, certain partially observed dynamical systems can be reconstructed via delay embeddings. Delay embedding of features was investigated in Hankel DMD~\cite{arbabi2017ergodic}, and reviewed in greater generality in~\cite{kamb2020time}. Our setup is different in that the features are themselves functions of delay embedded sample points. This leads to smaller matrices in~\eqref{eq:Psix}-\eqref{eq:Psiy} while improving distance measurements. 

Here, we define samples 
as delay embeddings of length $\ell$, 
\begin{align}\begin{split}\label{eq:time-embedding}
         \xb_{n+1} &= \begin{bmatrix} \xb(n\tau) & \ldots &  \xb((n+\ell-1)\tau)\end{bmatrix} \\
         \yb_{n+1} &= \begin{bmatrix} \xb((n+1)\tau) & \ldots &  \xb((n+\ell)\tau)\end{bmatrix},
         \end{split}
     \end{align}
    where $\xb(0)$ is some initial state. 
     The evolution map $\Fc_\tau$ extends 
     to such states in a natural way, and 
     the associated Koopman operator 
     is then defined on functions of time embedded states. From here on, we abuse notation by  
     writing $\xb$ or $\xb'$ for a delay embedding (or sample) of the form~\eqref{eq:time-embedding}.
     
     Note that the features~\eqref{eq:features} also have a delay-embedded structure, as in Hankel DMD~\cite{arbabi2017ergodic} 
     and kernel EDMD~\cite{kevrekidis2016kernel}. We find that an additional layer of embedding, applied to 
     the samples themselves, 
     improves distance measurements. This helps mitigate the curse of dimensionality and noise effects.
     

\subsection{The FKMD Algorithm}

FKMD combines ordinary KMD with a particular choice of features (equations~\eqref{eq:features} or~\eqref{eq:features-RFF}) and delay-embedded samples (equation~\eqref{eq:time-embedding}), along with iterations that improve the features based on matrix $\Mb$ (equations~\eqref{eq:grad-outerproduct},\eqref{eq:curvature}).  
The matrix $\Mb$ featurizes by enforcing isotropic changes in space and time.
We summarize in Algorithm~\ref{alg:FKMD} below, which we call Featurized Koopman Mode Decomposition (FKMD).

 \begin{algorithm}[FKMD]\label{alg:FKMD}
 Choose parameters $\ell$, $h>0$, and $R$, mode selection rules ${\mathcal S}$, and set $\Mb = \Ib$. Then, iterate steps $1$-$6$ below until convergence:
 \begin{itemize}
 \item[1.] Compute bandwith $\sigma$ from  $\Mb^{1/2}\xb_1,\ldots, \Mb^{1/2}\xb_N$; set
 \begin{equation*}
     \Mb \gets \Mb/(h\sigma)^2.
 \end{equation*}
        \item[2.]  Construct  $\Psib_{\xb}$ and $\Psib_{\yb}$ defined in~\eqref{eq:Psix}-\eqref{eq:Psiy}.
        
        \item[3.] Solve for $\Kb$ and $\Bb$ in~\eqref{eq:Koopman-solve}-\eqref{eq:B-solve}.
            \item[4.] Compute eigenvalues $\mu_m$ and eigenvectors $\xib_m$,$\wb_m$ of $\Kb$ as in~\eqref{eq:Koopman-decomposition}.
            \item[5.] Compute Koopman eigenfunction and modes using
            \begin{align*}
                 \phi_m(\xb) = \psib(\xb)\xib_m, \quad \vb_m^\ast = \wb_m^\ast \Bb.
            \end{align*}
        
            \item[6.]  Update $\Mb$ using~\eqref{eq:grad-outerproduct},~\eqref{eq:curvature};  replace $\Mb$ by its real part.
 \end{itemize}
 \end{algorithm}

Some comments are in order: 
\begin{itemize}
    \item At any iteration of FKMD, 
    we can make predictions using~\eqref{eq:estimate-observable}. Empirically, convergence of $\Mb$ and the predictions occurs after $<10$ iterations, and the predictions  improve monotonically with iteration. 
    
    \item  Replacing $\Mb$ by its real part in FKMD makes $\Mb$ into a symmetric positive semidefinite matrix, and it does not change the Mahalanobis distance $|\Mb^{1/2}(\xb-\xb')|$ or the kernel
$k_{\Mb}(\xb,\xb')$. 
 \item FKMD
needs only a few user-chosen parameters. 
This may be an advantage 
 over neural networks, which can have a much larger set of 
hyperparameters\cite{yamashita2018convolutional}.
\end{itemize}

 \subsection{Tuning} \label{sec:tuning}

FKMD requires the following user-chosen parameters: the number of iterations; $R$, the number of features; $\ell$, the delay embedding length; $h$, a scalar; $\sigma$, the bandwidth; and ${\mathcal S}$, the mode set for defining $\Mb$. We discuss these choices below.

We have found empirically that FKMD converges in a small number of iterations (at most $10$, and usually fewer than $5$). Of course, assuming sufficient data, results improve as the number, $R$, of features grows. Increasing the embedding length, $\ell$, can give better results in systems that are partly observed, assuming that $R$ is concomitantly increased. Regarding $h$, we found that 
the value $h = 1$ works well as a default, and use it in the experiments in Sections~\ref{sec:ODE} and~\ref{sec:Lorenz}. In Section~\ref{sec:cells} we refine this default value slightly using cross-validation.
 
In kernel learning, 
bandwidth is often based on the pairwise distances between samples~\cite{silverman2018density,flaxman2016bayesian}. In FKMD, it is more appropriate to use the transformed samples, $\Mb^{1/2}\xb_1,\ldots,\Mb^{1/2}\xb_N$. 
In the experiments in Sections~\ref{sec:ODE} and~\ref{sec:Lorenz}, $\sigma^2$ is the trace of the covariance matrix of $\Mb^{1/2}\xb_1,\ldots,\Mb^{1/2}\xb_N$. This bandwidth, based on the standard deviation of the samples, is similar to Silverman's rule of thumb~\cite{silverman2018density}. In the cell-signaling experiment of Section~\ref{sec:cells}, $\sigma$ is the (vector) standard deviations of the pairwise absolute differences between the transformed samples~\cite{navarra2021estimation}. These pairwise differences can be subsampled if needed. This reflects the need for larger bandwidth when there is a lot of variation in the pairwise differences.

The mode set ${\mathcal S}$ could be chosen from physical considerations (like desired timescales) or cross-validation (enforcing FKMD accuracy). 
We used cross-validation to 
choose ${\mathcal S}$ in all experiments. For the experiments in Sections~\ref{sec:ODE} and~\ref{sec:Lorenz}, we choose the top few modes based on magnitude; for the experiments in Section~\ref{sec:cells}, we also
use physical considerations to define cutoffs.

There are a couple of other (optional) regularizations that could be applied to FKMD. For example, ``unphysical'' modes could be removed from equation~\eqref{eq:estimate-observable}. Indeed, in the cell signaling experiments in Section~\ref{sec:cells}, we removed modes with $\textup{Re}(\lambda_m) \gg 0$ or {\em i.e.}, $|\textup{Im}(\lambda_m)| \gg 0 $; this corresponds to eliminating diverging and fast oscillatory modes.

We have also found that it may be useful to add a small ridge $\Mb \leftarrow \Mb + \delta \Ib$ to $\Mb$ after Step 6, where $\delta>0$ is a small parameter. This can prevent degradation of results with iteration.
In practice, with large sample sizes, subsampling may be used to estimate $\Mb$ via~\eqref{eq:grad-outerproduct}. While we define the initial $\Mb$ in FKMD as a scalar multiple of the identity matrix, it 
could also be based on 
initial knowledge of the system.

\subsection{Scaling up to larger sample size}\label{sec:scaling_up}

 On top of the difficulty of choosing good 
 features, kernel methods 
 have been plagued 
 by the computational 
 complexity of large linear solves~\cite{tropp2023randomized}, typically 
 limiting sample size to $N \le 10^5$. Using random Fourier features~\cite{rahimi2007random,yang2012nystrom,kammonen2020adaptive,nuske2023efficient,degennaro2019scalable}, however, we can scale FKMD to
 large sample size $N$:
\begin{equation}\label{eq:features-RFF}
    \psi_m^{RFF}(\xb) = \exp(i \omegab_m^\ast \Mb^{1/2}\xb),  \qquad m=1,\ldots,R.
\end{equation}
Here, $\omegab_m$ are iid Gaussians with mean $\Ob$ and covariance $\Ib$,
\begin{equation}\label{eq:wavenumbers}
\omegab_m \sim {\mathcal N}(\Ob,\Ib),
\end{equation}
and $\Mb$ is symmetric positive semidefinite. 

We expect good results with $R \ll N$, and this leads to much faster linear algebra routines. 
The features~\eqref{eq:features-RFF}-\eqref{eq:wavenumbers} 
essentially target the same linear system~\eqref{eq:Koopman-solve} as the kernel features, but they do it more efficiently by sampling; proof is in Appendix~\ref{appendix:RFF}.

\section{Experiments}

\subsection{Experimental setup}

{The experiments below are organized as follows. We begin with a high-dimensional system of linear differential equations that has effectively low-dimensional dynamics. In this example, we do not delay embed the samples, so that we can focus on the effect of $\Mb$. We illustrate a direct connection between $\Mb$ and the matrix driving the linear system (Theorem~\ref{thm:main}), and show that this leads to improved predictions.}

{Our second example is a high-dimensional Lorenz attractor with added noise, where a small percent of the system is observed. In this example, we needed substantially long delay embeddings to reliably forecast. We show that despite high dimensionality, presence of noise, and complex dynamics, we get accurate predictions when using sufficiently long delay embeddings and mapping them by $\Mb^{1/2}$.} 

{Our final example uses real-world cell-signaling data. There, we track single-cells through time, observe a scalar function of the cells (that is related to cellular behaviors like proliferation rate), and find that FKMD yields predictive capability 
hours into the future.}

{FKMD generalizes standard Koopman mode decomposition methods by adding two ingredients: a delay embedding of the sample points, and a learned mapping $\Mb^{1/2}$. Without these add-ons, FKMD is  EDMD~\cite{williams2015data} with Hankel data matrices~\cite{arbabi2017ergodic} formed from Gaussian kernel features~\cite{williams2014kernel} (or approximated by random Fourier features). 
We refer to this as ``kernel EDMD;'' this method 
has been used {\em e.g} in~\cite{kevrekidis2016kernel,klus2020kernel,philipp2023error}.}

{There are a host of KMD methods~\cite{colbrook2023multiverse}, and it is outside the scope of this paper to compare our method against all of them. However, viewing FKMD as an extension of kernel EDMD leads to natural comparisons, in which either or both of the add-ons -- delay embedding of the samples, and mapping samples by $\Mb^{1/2}$ -- are not included. }

{Without delay embedding of the samples, we were unable to forecast reliably in the Lorenz and cell signaling experiments. 
We therefore only compare with basic kernel EDMD in our simplest example system, the system of linear differential equations.
We do, however, illustrate the failure to forecast on the Lorenz system in Figure~\ref{fig:KMD}.}

{This leads us to a comparison with kernel EDMD with delay embedded samples, or equivalently, FKMD with $\Mb = \Ib$.  
This comparison is automatically included in our results below, as it corresponds to the first iteration of FKMD. From this point of view, FKMD further refines kernel EDMD 
via learning of $\Mb$.}

\subsection{High-dimensional system of differential equations}
\label{sec:ODE}

{In this section, we demonstrate FKMD on a system of linear differential equations with oscillating 
and decaying components:
\begin{equation}\label{eq:ODEmodel}
    \frac{d\xb(t)^\ast}{dt} = \xb(t)^\ast \Ab.
\end{equation}
The (real) matrix $\Ab$, shown in Figure~\ref{fig:oscillator}(b), has eigenvalues  $\pm \sqrt{5}i$, $\pm \sqrt{2}i$, and $-3$, along with $25$ eigenvalues $\approx -10^{-2}$. This corresponds to $2$ oscillating modes, one fast decay mode, and many slow decay modes. We will take observations of the entire system, so that  delay embedding of the sample points is not needed. This allows us to  focus on the effect of the $\Mb$ matrix for forecasting.}

{Sample inputs, $\xb_n$, are independent draws from a standard normal distribution. Sample outputs, $\yb_n$, are obtained from $\xb_n$ by integrating~\eqref{eq:ODEmodel} up to time $\tau = 10^{-2}$ using $4$th order Runge-Kutta~\cite{butcher1996history} with integrator time step $10^{-3}$. Data is divided into a training and testing set. We use FKMD with random Fourier features~\eqref{eq:features-RFF}-~\eqref{eq:wavenumbers} to forecast on the testing set, given its initial state; the forecast looks ahead $100$ discrete time steps of size $\tau$.}

\begin{figure}[!ht]
  \centering
  \subfloat[][]{\includegraphics[width=.43\textwidth]{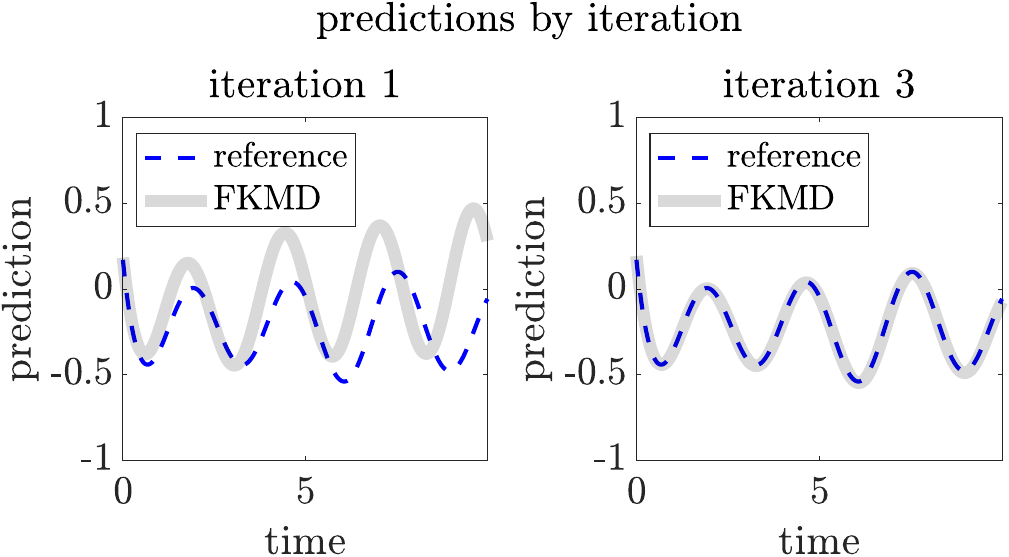}} \\ 
  \subfloat[][]{\includegraphics[width=.48\textwidth]{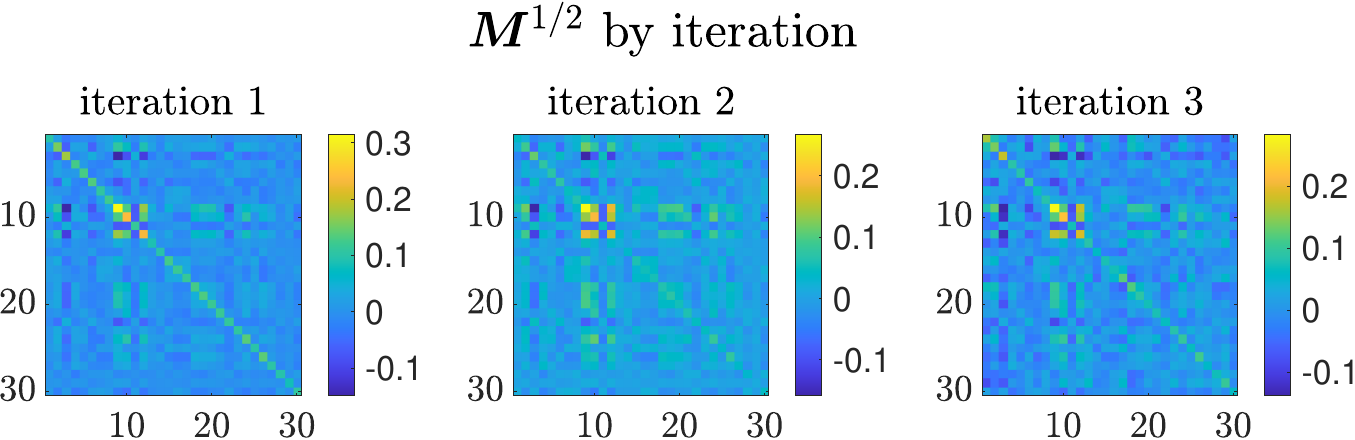}}\\
  \subfloat[][]{\includegraphics[width=.22\textwidth]{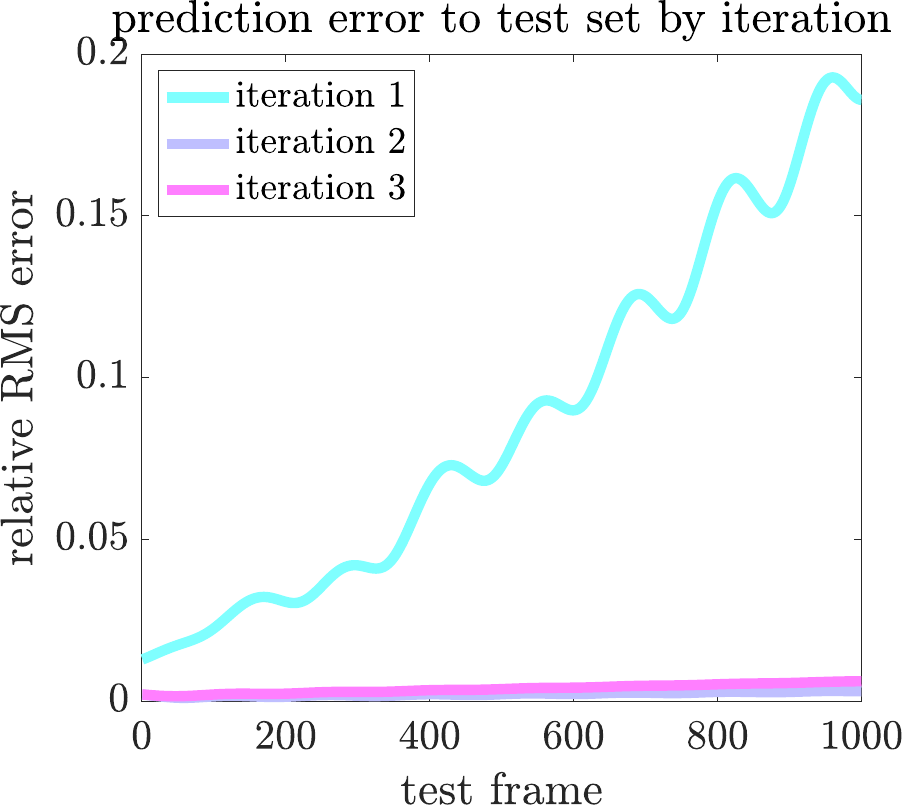}}\quad
  \subfloat[][]{\includegraphics[width=.23\textwidth]{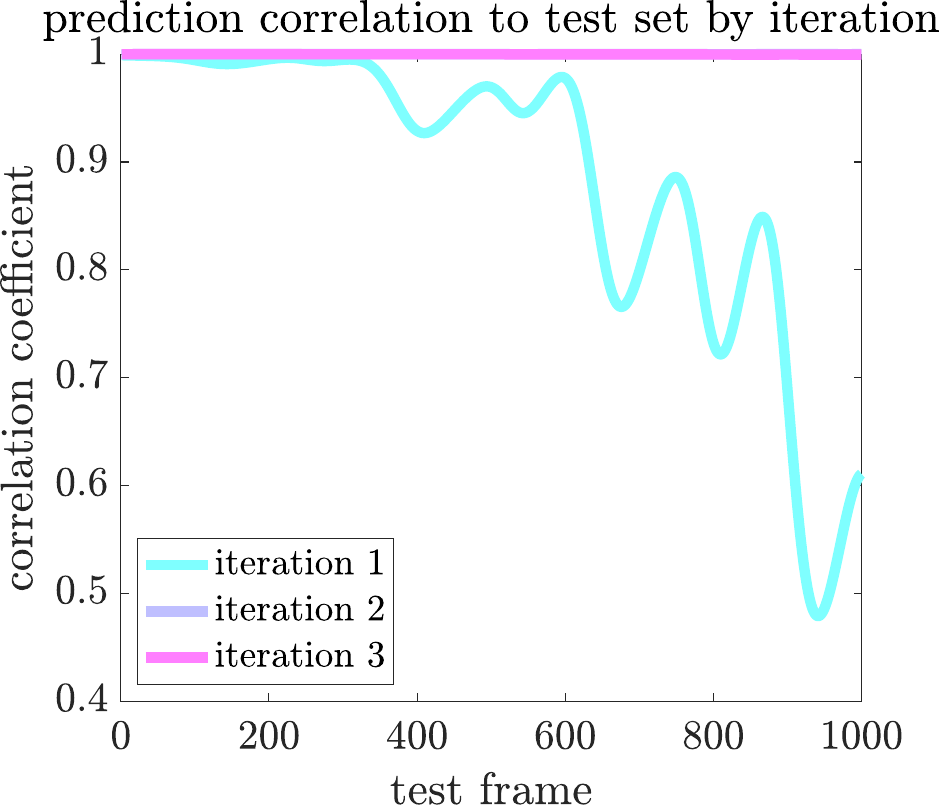}} 
      \caption{{Results from the linear system of differential equations in  Section~\ref{sec:ODE}: (a) Predictions of the first coordinate of $\xb(t)$ from the first and third iteration of FKMD. Plots are similar for other coordinates. (b) $\Mb$ matrix at each iteration.  (c)-(d) Predictions errors and correlations on a test set. FKMD converges by the second iteration, producing very accurate results. (Curves for the $2$nd and $3$rd iterations mostly overlap.) As explained above, the first iteration is kernel EDMD~\cite{kevrekidis2016kernel}. }}
  \label{fig:oscillator}
\end{figure}

{For the FKMD parameters, we use $N = 10^5$ training sample points, $R = 10^3$ features, delay embedding of length $\ell = 1$ ({\em i.e.}, the sample points are not delay embedded), and $h = 1$. 
We observe the full system, $\gb(\xb) = \xb^\ast$. We use the top $400$ modes and a random subsample of $5000$ points to define $\Mb$; results are not very sensitive to these choices so long as the parameters are large enough.}

{Figure~\ref{fig:oscillator} shows the results. FKMD accurately estimates $\Mb$ after the first iteration and converges by the $2$nd iteration, shown in Figure~\ref{fig:oscillator}(a). Improved forecasting with iteration -- see Figure~\ref{fig:oscillator} (a) -- can be traced to how the $\Mb$ matrix makes the system look isotropic, as we describe in Theorem~\ref{thm:main} and illustrate in Figure~\ref{fig:oscillator_match}. We measure relative root mean squared (RMS) error and correlations of the 
predictions in Figure~\ref{fig:oscillator}(c)-(d). (The relative RMS error is the root mean squared error divided by the average norm of the data.) The error and correlations markedly improve after the first iteration, and remain steady thereafter.}

\begin{figure}[!ht]
  \centering
{\includegraphics[width=.48\textwidth]{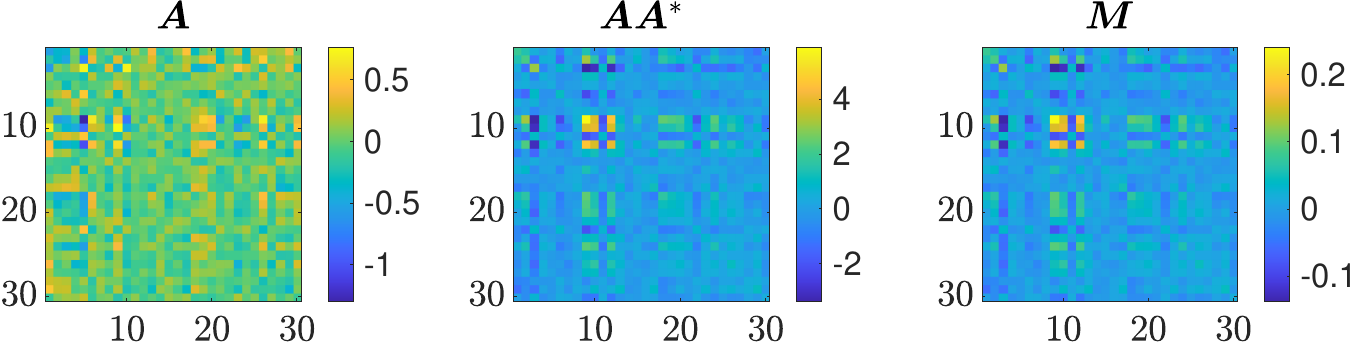}}
      \caption{{Matrices $\Ab$, $\Ab \Ab^\ast$, and $\Mb$ from the experiment in Section~\ref{sec:ODE}. After FKMD convergence,  $\Mb$ matches $\Ab \Ab^\ast$ up to a scalar factor; see Theorem~\ref{thm:main}. }}
  \label{fig:oscillator_match}
\end{figure}

\subsection{Lorenz attractor}\label{sec:Lorenz}

Next, we apply FKMD to data from the Lorenz-96 model~\cite{lorenz1996predictability}, a high-dimensional ODE exhibiting chaotic behavior. This model (and
its 3-dimensional predecessor~\cite{lorenz1963deterministic}) are
 often used to interpret atmospheric convection and to test tools in climate analysis~\cite{hu2021particle}. The model we use is 
\begin{equation}\label{eq:lorenz96}
    \frac{d\theta_j}{dt} = (\theta_{j+1}-\theta_{j-2})\theta_{j-1}-\theta_j + F,
\end{equation}
with $j = 1,\ldots,40$ periodic coordinates ($j \equiv j\textup{ mod }40$). We set $F = 8$, and generate a long trajectory of data by integrating~\eqref{eq:lorenz96} using $4$th order Runge-Kutta~\cite{butcher1996history} with integrator time step $10^{-2}$ and initial condition
$$\theta_j(0) = \begin{cases}F + 1, & j \textup{ mod }5 =0 \\ F, &\textup{else}\end{cases}.$$

To illustrate the power of 
FKMD, we observe just $2.5\%$ of the system, namely the first coordinate $\theta_1$, and we add nuisance variables.
Specifically, we use the delay embedding~\eqref{eq:time-embedding} with $\tau = 0.05$ and 
\begin{equation}\label{eq:lorenz-x}
    \xb(n\tau) = \begin{bmatrix} \theta_1(n\tau) & \textup{noise}(n\tau)\end{bmatrix},
\end{equation}
where $\textup{noise}(n\tau)$ for $n=0,1,2,\ldots$ are independent Gaussian random variables with mean $0$ and standard devation equal to $3$ (commensurate with that of $\theta_1$). {We divide the time series data into a training and testing set, and 
use FKMD with random Fourier features to forecast on the testing set, given its initial state. The testing set consists of the last $100$ time steps of the data.}

\begin{figure}[!ht]
  \centering
  \subfloat[][]{\includegraphics[width=.43\textwidth]{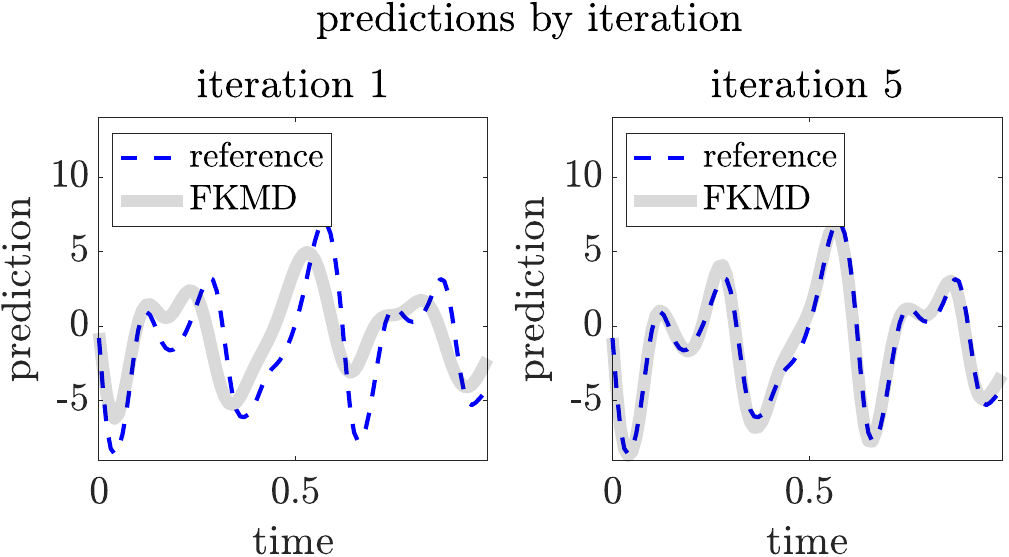}} \\
  \subfloat[][]{\includegraphics[width=.48\textwidth]{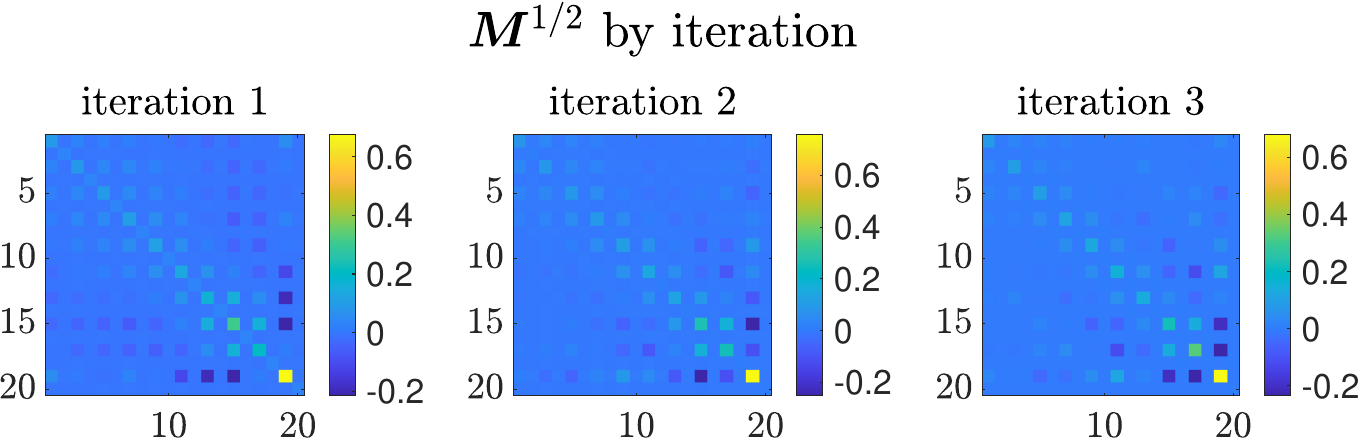}}\\
  \subfloat[][]{\includegraphics[width=.22\textwidth]{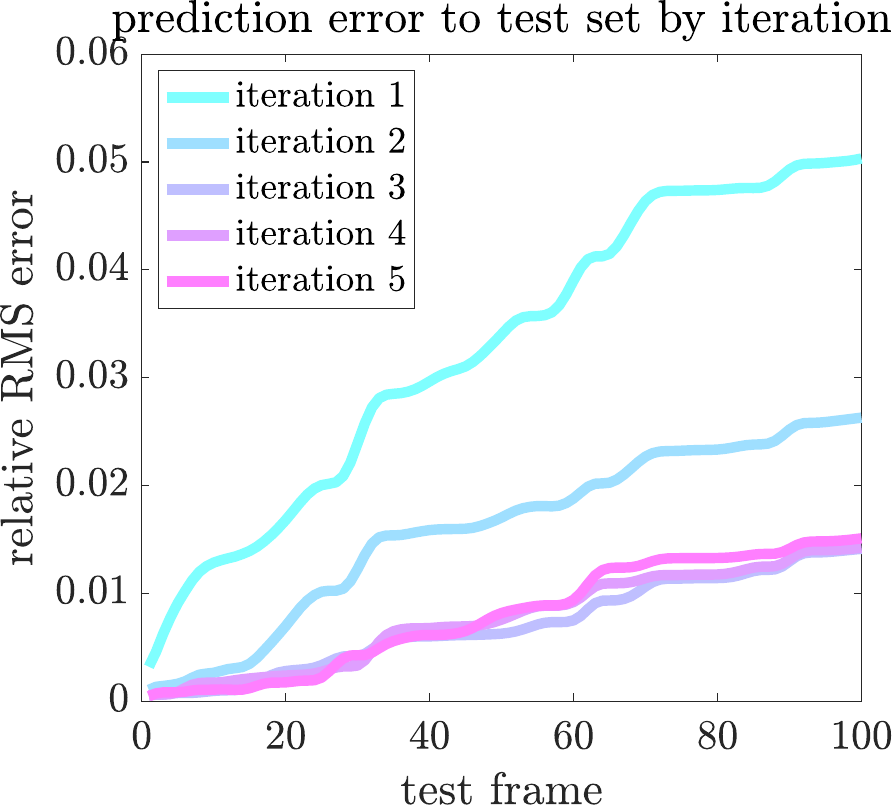}}\quad
  \subfloat[][]{\includegraphics[width=.235\textwidth]{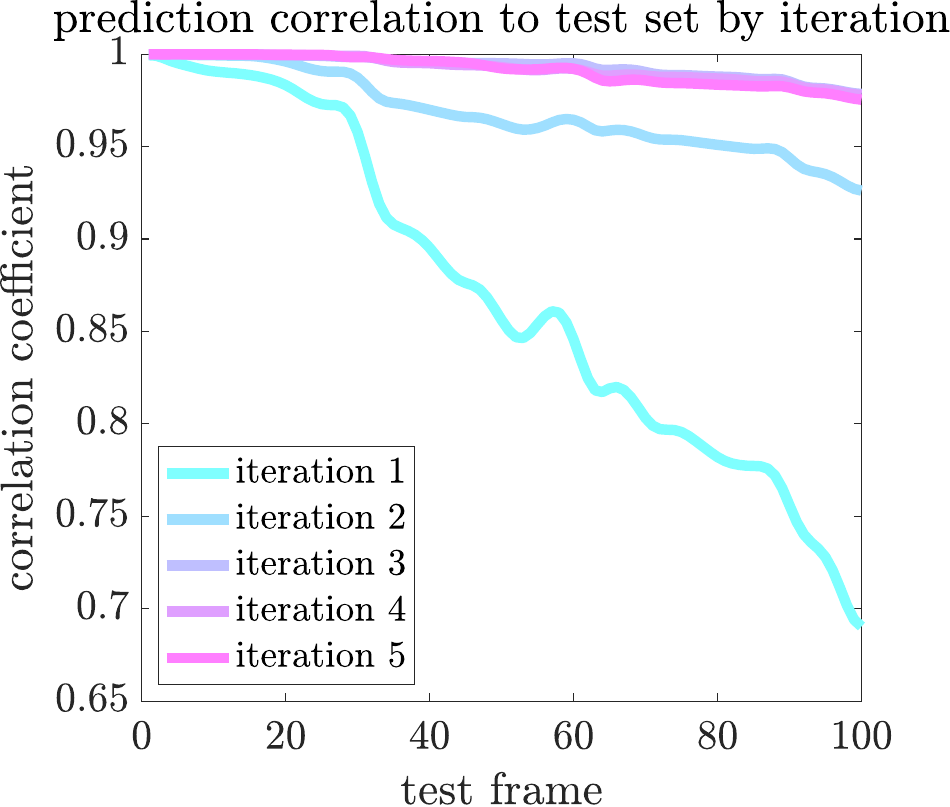}}
  \caption{{Results from the Lorenz system in Section~\ref{sec:Lorenz}: (a) Predictions of $\theta_1(t)$ from the first and fifth iteration of FKMD. (b) $\Mb$ matrix at iterations $1$, $2$ and $3$. Plotted is the bottom $20\times 20$ submatrix of $\Mb$, corresponding to the embedded coordinates closest to the ``current'' time. Mapping by $\Mb^{1/2}$ sends the noise variables to $0$; this is apparent from the checkerboard pattern of $0$s in the matrix. (c)-(d) Prediction errors and correlations compared to test set. (Curves for iterations $3$-$5$ are mostly overlapping.) The matrix $\Mb$ and the FKMD predictions visibly converge after $3$ iterations.}}
  \label{fig:lorenz}
\end{figure}

For the FKMD parameters, we use $N = 10^6$ sample points, $R = 5000$ features, a delay embedding of length $\ell = 100$, and $h=1$. The observation $\gb(\xb)$ is a $1 \times 200$ vector associated with delay embeddings of~\eqref{eq:lorenz-x} as defined in~\eqref{eq:time-embedding}. 
We use the top $20$ 
modes and a random subsample of $5000$ points to define $\Mb$;
results are not overly sensitive to these choices, but some mode cutoff is necessary to eliminate the effect of the nuisance coordinates.

Results are plotted in Figure~\ref{fig:lorenz}. 
Figure~\ref{fig:lorenz}(a) shows inference 
using equation~\eqref{eq:estimate-observable}. FKMD provides a very close match to the reference after $3$ iterations. Figure~\ref{fig:lorenz}(b) shows the change of variables matrix. {The $\Mb^{1/2}$ mapping eliminates the nuisance 
coordinates -- leading to a checkerboard pattern in Figure~\ref{fig:lorenz}(b) -- while preserving 
the structure of the underlying signal. For better visibility, only the bottom $20\times 20$ submatrix of $\Mb^{1/2}$ is shown. This portion of the matrix corresponds to the embedding coordinates closest to the ``current'' time.
Figure~\ref{fig:lorenz}(c) shows relative RMS error and correlations against a test set. 
We compute these quantities with respect to the noise-free trajectory. FKMD provides an 
excellent fit to the data by iteration $3$.}

\begin{figure}[!ht]
  \centering
{\includegraphics[width=.21\textwidth]{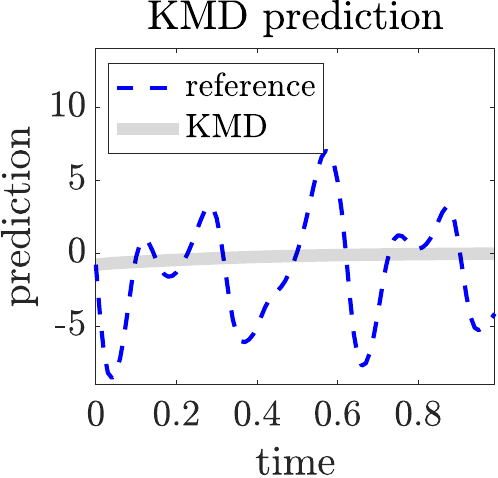}}
  \caption{{Kernel EDMD~\cite{kevrekidis2016kernel} is unable to make predictions that correlate with the test set for the Lorenz example of Section~\ref{sec:Lorenz}.}}
  \label{fig:KMD}
\end{figure}

{Results from kernel EDMD~\cite{kevrekidis2016kernel} are shown in Figure~\ref{fig:KMD}. There, 
we use the same parameters from 
FKMD ($N = 10^6$ samples, $R = 5000$ features, $h =1$) but 
do not delay embed the sample points, and do not use matrix $\Mb$. Kernel EDMD is not able to accurately forecast on the Lorenz data.}

\subsection{Cell signaling dynamics}\label{sec:cells}

In a real-world data-driven setting, complex and potentially noisy temporal outputs derived from measurement may not obey a simple underlying ODE or live on a low-dimensional dynamical attractor. Information contained by internal signaling pathways within living cells is one such example, being complex and subject to noisy temporal outputs arising from properties of the system itself and experimental sources.

With this in mind, we next apply FKMD to dynamic signaling activity in cancer cells to assess its performance. 
We show that our methods enable the forward prediction of single-cell signaling activity from past knowledge in a system where signaling is highly variable from cell to cell and over time\cite{davies2020systems}. The extracellular signal‑regulated kinases (ERK) signaling pathway is critical for the perception of cues outside of cells and for translation of these cues into cellular behaviors such as changes in cell shape, proliferation rate, and phenotype\cite{copperman2023morphodynamical}. Dynamic ERK activity is monitored via the nuclear or cytoplasmic localization of the fluorescent reporter (Figure~\ref{fig:cells}A). We track single-cells through time in the live-cell imaging yielding single-cell ERK activity time series (Figure~\ref{fig:cells}C). The first 72 hours of single-cell trajectories serve as the training set to estimate the Koopman operator, and we withhold the final 18 hours of the single-cell trajectories to test the predictive capability of FKMD. The raw ERK activity trajectories on their own yield no predictive capability via standard Kernel DMD methods, but our iterative procedure to extract the matrix $\Mb$ leads to a coordinate rescaling which couples signaling activity across delay times (Figure~\ref{fig:cells}B,D) and enables a forward prediction of ERK activity across the testing window (Figure~\ref{fig:cells}D). {Many higher-frequency changes in ERK activity appear stochastic and are not forecasted, but the FKMD method reveals a slow ($\sim 8$ hours) and predictable component to the ERK signaling activity (Figure~\ref{fig:cells}C).}

We use $N = 5202$ samples and kernel features with $R = N$, and we choose bandwidth $h = 1.05$ and a delay embedding of length $\ell = 49$. The function $\gb(\xb)$ is a $1 \times 49$ delay embedding of the scalar ERK activity. For inference, we exclude modes  where $\textup{Re}(\lambda_m) >  0.15$ and $|\textup{Im}(\lambda_m)| >  \pi/3$. This amounts to excluding unstable modes and modes that oscillate quickly. We use the remaining modes to construct $\Mb$. Prediction quality is quantified by estimating the relative error and correlation between inferred and test set ERK activity trajectories.

\begin{figure}[h]
         \centering
\includegraphics[width=0.5\textwidth]{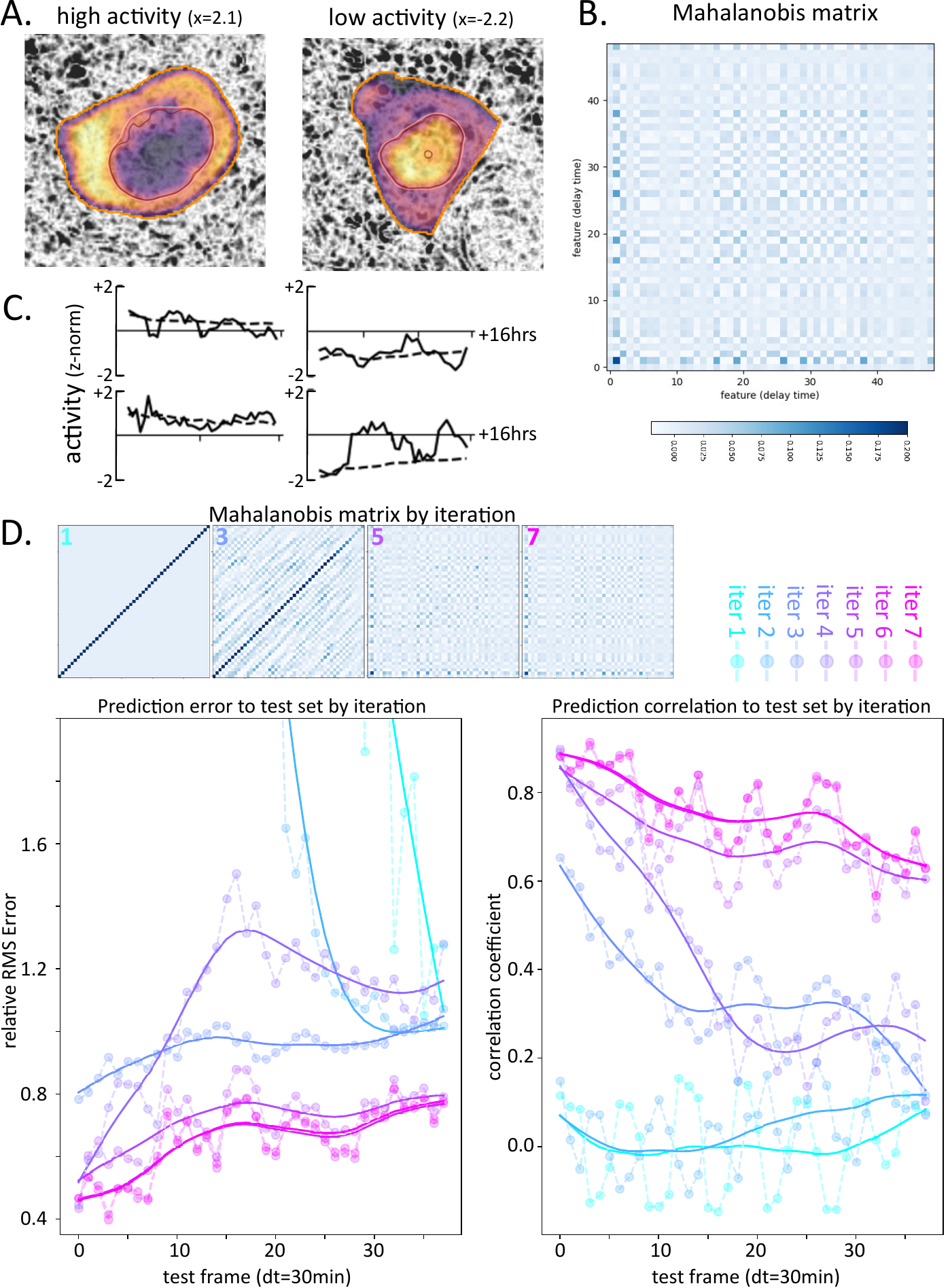}
\caption{FKMD-enabled cell signaling state prediction from Section~\ref{sec:cells}: (A) Fluorescent ERK reporter expressing breast cancer cell embedded in a mammary tissue organoid, showing representative high activity (left, cytoplasmic localized) and low activity (right, nucleus localized). (B) Malahanobis matrix at iteration 8, where test set correlation is maximized. (C) Representative single-cell ERK activity traces, measured test set (solid lines), and FKMD-predicted (dashed lines). (D) matrix $\Mb^{1/2}$ by FKMD iteration (top), and FKMD prediction performance from 0-18hrs quantified by relative 
root mean squared (RMS) error and correlation to test set (bottom left and right) by iteration number (cyan to magenta circles). solid lines are spline fits added as guides to the eye.}
        \label{fig:cells}
\end{figure}

\section{Discussion and future work}

This article introduces FKMD, 
a method we propose that can generate more accurate predictions than ordinary kernel KMD~\cite{kevrekidis2016kernel}. The method 
is based on delay embeddings 
and a Mahalanobis distance that helps mitigates the curse of dimensionality.
{Results in three experiments -- a  system of linear differential equations, a high dimensional Lorenz system,  and a cell signaling problem -- 
illustrate the promise of the method for forecasting in complex systems. 
In all these experiments, 
using $\Mb$ improved predictions compared to kernel EDMD (Figures~\ref{fig:oscillator},~\ref{fig:lorenz} and~\ref{fig:cells}). In the Lorenz and cell-signaling problems, using delay 
embedded samples was essential to obtain any reasonable forecast in kernel EDMD (Figure~\ref{fig:KMD}). 
We conclude that FKMD can 
be a substantial improvement over kernel EDMD when the underlying dynamics are effectively low-dimensional, or when the system can only be partially observed.}

Many theoretical and algorithmic questions remain. 
Empirically, we have found that a few iterations 
and modes can lead to good results, 
but more empirical testing is needed, 
and our theoretical understanding of 
these issues is lacking. For example, 
we cannot yet describe a simple 
set of conditions that 
guarantees good behavior of Algorithm~\ref{alg:FKMD}, like convergence to a fixed point. 

We would also like to explore alternative methods for scaling 
up FKMD to larger sample sizes. 
We use random Fourier features, but other possibilities come from modern advances in randomized numerical linear algebra, {\em e.g.} randomly pivoted Cholesky~\cite{tropp2023randomized,chen2022randomly}. Such methods promise spectral efficiency for solving symmetric positive definite linear systems. Assuming fast spectral decay of $\Psib_{\xb}$, these techniques could help our methods scale to even larger sample sizes.
 We will explore the application of these cutting-edge methods in future works.

\begin{acknowledgments}
D. Aristoff gratefully acknowledges support from the National Science Foundation via Award No. DMS 2111277. N. Mankovich acknowledges the support of Generalitat Valenciana and the Conselleria d’Innovació, Universitats, Ciència i Societat Digital, through the project ``AI4CS: Artificial Intelligence for complex systems: Brain, Earth, Climate, Society'' (CIPROM/2021/56). This project was supported in part by funding (2023-1713) from the Cancer Early Detection Advanced Research Center at Oregon Health \& Science University’s Knight Cancer Institute (Jeremy Copperman and Alexander Davies). Alexander Davies gratefully acknowledges support from NIH/ORIP award K01OD031811.

D. Aristoff and J. Copperman acknowledge enlightening discussions about the gradient outerproduct\cite{radhakrishnan2022feature} with G. Simpson and R.J. Webber.

\end{acknowledgments}
\appendix


\section{Derivation of Koopman eigendecomposition}\label{appendix:inference}

Here, we show how to arrive at the Koopman eigendecomposition~\eqref{eq:estimate-observable}. This has been shown already in~\cite{williams2015data}, but we provide a streamlined derivation here for convenience.

Recall that the matrix $\Kb$ is a finite-dimensional approximation to the Koopman operator. This approximation is obtained by 
applying a change of variables from 
sample space to feature space. The change of variables is given by the matrix $\Psib_{\xb}$. This leads to the following equation for inference: 
\begin{equation}\label{eq:inference}
    \begin{bmatrix}\Kc_\tau (\gb)(\xb_1) \\ \vdots \\ \Kc_\tau (\gb)(\xb_N)\end{bmatrix} \approx \Psib_{\xb} \Kb \Psib_{\xb}^{\dag} \begin{bmatrix}\gb(\xb_1) \\ \vdots \\ \gb(\xb_N) \end{bmatrix},
\end{equation}
where $\dag$ denotes the Moore-Penrose pseudoinverse. Similarly,
\begin{equation*}
    \Bb = \Psib_{\xb}^{\dag}\begin{bmatrix}\gb(\xb_1) \\ \vdots \\ \gb(\xb_N) \end{bmatrix}.
\end{equation*}
The 
eigendecomposition of $\Kb$ can be written as
\begin{equation}\label{eq:K-decomposition}
    \Kb = \Xib \Db \Wb^\ast,
\end{equation}
where $\Kb \Xib = \Xib \Db$ and $\Wb^\ast \Kb = \Db \Wb^\ast$, and we may assume that $\Wb^\ast \Xib = \Ib$. 
Here, $\Db$ is the diagonal matrix of Koopman eigenvalues, $\mu_m$; 
that is, $\Db = \exp(\tau \Lambdab)$ where $\Lambdab$ is the diagonal matrix of continuous time Koopman eigenvalues, $\lambda_m$.

Plugging~\eqref{eq:K-decomposition} into~\eqref{eq:inference}, 
\begin{equation}\label{eq:Ktau}
     \begin{bmatrix}\Kc_\tau (\gb)(\xb_1) \\ \vdots \\ \Kc_\tau (\gb)(\xb_N)\end{bmatrix} \approx \Psib_{\xb} \Xib \exp(\tau\Lambdab) \Wb^\ast \Bb.
\end{equation}
The definition of Koopman modes and Koopman eigenfunctions shows that the rows of $\Wb^\ast \Bb$ are the Koopman modes $\vb_m^\ast$, while the Koopman eigenfunctions are sampled by the columns of
\begin{equation}\label{eq:PsiXi}
    \Psib_{\xb}\Xib = \begin{bmatrix} \phi_1(\xb_1) & \ldots & \phi_R(\xb_1) \\ 
    \vdots & & \vdots \\
    \phi_1(\xb_N) & \ldots & \phi_R(\xb_N)\end{bmatrix}.
\end{equation}
Substituting~\eqref{eq:PsiXi} into~\eqref{eq:Ktau} and writing the matrix multiplication in terms of outer products yields equation~\eqref{eq:estimate-observable}, provided we substitute $\xb_n$ for $\xb$, using any sample point $\xb_n$.

\section{Choice of matrix}\label{appendix:Mahalanobis}

Here, we explain the reasoning behind the choice of $\Mb$ in more detail.
Recall that the matrix defines a change of variables, 
$\tilde{\xb} = \Mb^{1/2} \xb$, where the tilde notation indicates
the changed variables.

In this appendix, {we assume  that $\Mb$ is symmetric positive definite -- in particular, invertible -- and we assume appropriate smoothness so that all the calculations make sense.}

Write $\tilde{\xb}_n = \Mb^{1/2}\xb_n$, $\tilde{\yb}_n = \Mb^{1/2}\yb_n$, and  
\begin{align*}
    \tilde{\gb}(\xb) = \gb(\Mb^{-1/2}\xb), \quad \tilde{\Fc}_\tau(\xb) = \Mb^{1/2}\Fc_\tau(\Mb^{-1/2}\xb). 
\end{align*}
Observe that then $\tilde{\gb}(\tilde{\xb}) = \gb(\xb)$ and $\tilde{\Fc}_\tau(\tilde{\xb}_n) = \tilde{\yb}_n$. We summarize these notations and mappings in Figure~\ref{fig:KMD_tildes}.
\begin{figure}[t!]
    \centering
    \includegraphics[width = \linewidth]{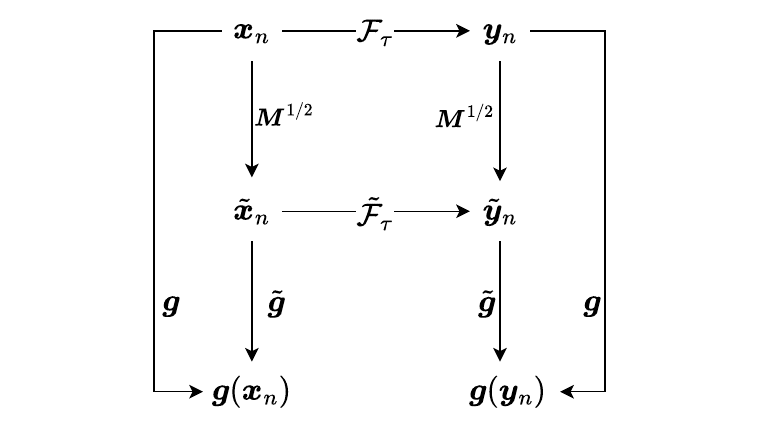}
    \caption{A figure summarizing the ``tilde'' notation for~\ref{propB}.}
    \label{fig:KMD_tildes}
\end{figure}
Define
\begin{align*}
    \Jb(\xb) &= \lim_{\tau \to 0} \tau^{-1}[\nabla (\gb \circ \Fc_\tau)(\xb) - \nabla \gb(\xb)], \\
     \tilde{\Jb}(\xb) &= \lim_{\tau \to 0}\tau^{-1}[\nabla (\tilde{\gb} \circ \tilde{\Fc}_\tau)(\xb) - \nabla \tilde{\gb}(\xb)].
\end{align*}

The next result,  Proposition~\ref{propB}, justifies our choice of $\Mb$. It shows that the 
changes in space and time 
in the transformed variables,
as measured by $\tilde{\Jb}$, are 
isotropic. This is a general result that is true regardless of kernel choice and other KMD hyperparameters.

\begin{proposition}\label{propB}
If $\Mb$ defined by~\eqref{eq:grad-outerproduct}-\eqref{eq:Jexact} is invertible, then
\begin{equation*}
   \frac{1}{N}\sum_{n=1}^N \left|\ub^\ast \tilde{\Jb}(\tilde{\xb}_n)\right|^2 \equiv 1, \quad \textup{for unit }\ub.
    \end{equation*}

    \end{proposition}

\begin{proof} 
By the chain rule, 
\begin{align*}
    &\nabla (\tilde{\gb}\circ \tilde{\Fc}_\tau)(\tilde{\xb}_n) \\
    &= \nabla \tilde{\Fc}_\tau(\tilde{\xb}_n)\nabla \tilde{\gb}(\tilde{\Fc}_\tau(\tilde{\xb}_n)) \\
    &= \Mb^{-1/2} \nabla \Fc_\tau(\Mb^{-1/2} \tilde{\xb}_n) \Mb^{1/2} \Mb^{-1/2}\nabla {\gb}(\Mb^{-1/2}\tilde{\yb}_n) \\
    &= \Mb^{-1/2} \nabla \Fc_\tau(\xb_n) \nabla \gb(\yb_n) \\
    &= \Mb^{-1/2}\nabla (\gb \circ \Fc_\tau)(\xb_n).
\end{align*}
Similarly, 
\begin{align*}
    \nabla \tilde{\gb}(\tilde{\xb}_n) &= \Mb^{-1/2}\nabla \gb(\Mb^{-1/2} \tilde{\xb}_n) \\
    &= \Mb^{-1/2} \nabla \gb (\xb_n).
\end{align*}
As a result,
\begin{equation*}
    \tilde{\Jb}(\tilde{\xb}_n) = \Mb^{-1/2} \Jb(\xb_n).
\end{equation*}
It follows that
\begin{align*}
    &\frac{1}{N}\sum_{n=1}^N \left|\ub^\ast \tilde{\Jb}(\tilde{\xb}_n)\right|^2 \\
    &= \frac{1}{N}\sum_{n=1}^N  \left|\ub^\ast \Mb^{-1/2}\Jb(\xb_n)\right|^2 \nonumber \\
    &= \frac{1}{N}\sum_{n=1}^N \ub^\ast \Mb^{-1/2}\Jb(\xb_n)\Jb(\xb_n)^\ast \Mb^{-1/2} \ub \equiv  1.
\end{align*}
\end{proof}

We investigate the linear case in  Proposition~\ref{rem: curvature linear} below. 

\begin{proposition}\label{rem: curvature linear}
Suppose 
that $\Fc_\tau$ is the evolution map of a linear ODE driven by a real invertible matrix $\Ab$,
$$\frac{d\xb(t)^\ast}{dt} =  \xb(t)^\ast \Ab,$$
and that 
$\gb(\xb) = \xb^\ast \Cb$, where $\Cb$ is invertible. 
Then
\begin{align*}
    \Jb &= \Ab \Cb, \\
    \Mb &= \Ab \Cb \Cb^\ast \Ab^\ast, \\
    \tilde{\Jb} &= (\Ab \Cb \Cb^\ast \Ab^\ast)^{-1/2}\Ab \Cb.
\end{align*}
In particular, $\tilde{\Jb}$ 
is an orthogonal matrix. 
\end{proposition}

\begin{proof}
Since $\Fc_\tau(\xb) =  e^{\tau \Ab^\ast}\xb$ and $\gb(\xb) = \xb^\ast \Cb$, we have $$\nabla \Fc_\tau(\xb) = e^{\tau \Ab}, \qquad \nabla {\gb}(\xb) = \Cb.$$ 
It follows that
\begin{align*}
    \Jb(\xb) &= \lim_{\tau \to 0}\tau^{-1}[\nabla (\gb\circ \Fc_\tau)(\xb)-\nabla \gb(\xb)] \\
    &= \lim_{\tau \to 0}\tau^{-1} (e^{\tau \Ab}- \Ib)\Cb = \Ab\Cb.
\end{align*}
By~\eqref{eq:grad-outerproduct},  $\Mb = \Ab \Cb \Cb^\ast \Ab^\ast$. Similarly, 
\begin{align*}
      \tilde{\Jb}(\xb) &= \lim_{\tau \to 0}\tau^{-1}[\nabla (\tilde{\gb}\circ \tilde{\Fc}_\tau)(\xb)-\nabla \tilde{\gb}(\xb)] \\
    &= \lim_{\tau \to 0}\tau^{-1}\Mb^{-1/2}(e^{\tau\Ab} - \Ib) \Cb = \Mb^{-1/2} \Ab\Cb.
\end{align*}
Finally, 
\begin{align*}
    \tilde{\Jb}^\ast \tilde{\Jb} = \Cb^\ast \Ab^\ast (\Ab \Cb \Cb^\ast \Ab^\ast)^{-1} \Ab \Cb = \Ib.
\end{align*}
\end{proof}



Propositions~\ref{propB}-\ref{rem: curvature linear} and the examples above explain the choice of $\Mb$ except for the variable scalar bandwidth $\sigma$. Computing $\sigma$ from pairwise distances is standard, except that in Algorithm~\ref{alg:FKMD} it is applied 
to the transformed samples, 
$\Mb^{1/2}\xb$, to appropriately reflect the change of variables. The additional constant scaling factor 
$h$ can be chosen using standard 
techniques such as cross validation~\cite{nuske2023finite}.

\begin{remark}
    We could have used another kernel that incorporates our change of variables, {\em e.g.}, the Laplace kernel
    \begin{equation*}
    k_{\Mb}^{laplace}(\xb,\xb') = \exp\left[-\left[(\xb-\xb')^*\Mb(\xb-\xb')\right]^{1/2}\right].
    \end{equation*}
Propositions~\ref{propB} and~\ref{rem: curvature linear} are independent of kernel choice.
\end{remark}


\section{Connecting kernels with random Fourier features}\label{appendix:RFF}

Below, we assume that $\Mb$ is symmetric positive definite.
The connection between the kernel features~\eqref{eq:features}-~\eqref{eq:kernel} and random Fourier features~\eqref{eq:features-RFF}-~\eqref{eq:wavenumbers} is the following.
\begin{proposition}\label{propC}
We have
\begin{equation*}
    k_{\Mb}(\xb,\xb') = {\mathbb E} \left[\psi_m^{RFF}(\xb)^\ast \psi_m^{RFF}(\xb')\right]
\end{equation*}
where ${\mathbb E}$ denotes expected value.
\end{proposition}

\begin{proof}
Let $\deltab = \xb'-\xb$, $\tilde{\deltab} =\Mb^{1/2} \deltab$. By completing the square, 
\begin{align*}
    &-\frac{1}{2}|\omegab|^2+i\omegab^\ast \Mb^{1/2} \delta = -\frac{1}{2}\left[(\omegab - i\tilde{\deltab})^T (\omegab - i\tilde{\deltab})\right] - \frac{1}{2}|\tilde{\deltab}|^2,
\end{align*}
so if samples live in $d$-dimensional (real) space ${\mathbb R}^d$, we get
\begin{align*}
   &{\bf E}\left[ \psi_m^{RFF}(\xb)^\ast \psi_m^{RFF}(\xb') \right]  \\
   &= (2\pi)^{-d/2}\int \exp(-|\omegab|^2/2)\exp(i\omegab^\ast\Mb^{1/2}\deltab)\,d\omegab\\
   &= \exp(-|\tilde{\delta}|^2/2) = k_{\Mb}(\xb,\xb').
\end{align*}
\end{proof}

Based on Proposition~\ref{propC},  we now show the connection between FKMD procedures with kernel and random Fourier features. Let $\Psib_{\xb}^{RFF}$ and $\Psib_{\yb}^{RFF}$ be the $N \times R$ samples by features matrices associated to random Fourier features~\eqref{eq:features-RFF}, and let $\Psib_{\xb}$ and $\Psib_{\yb}$ be the same matrices associated with kernel features~\eqref{eq:features}. Using Proposition~\ref{propC}, for large $R$,
\begin{equation}\label{eq:innerprod}
    \Psib_{\xb} \approx \Psib_{\xb}^{RFF} (\Psib_{\xb}^{RFF})^\ast, \qquad \Psib_{\yb} \approx \Psib_{\yb}^{RFF} (\Psib_{\xb}^{RFF})^\ast.
\end{equation}
Assume the columns of $\Psib_{\xb}^{RFF}$ are linearly independent. Then 
\begin{equation}\label{eq:id}
(\Psib_{\xb}^{RFF})^\ast[(\Psib_{\xb}^{RFF})^\ast]^\dag = \Ib
\end{equation}
where $\dag$ is the Moore-Penrose pseudoinverse. Define
\begin{equation}\label{eq:KRFF}
    \Kb^{RFF} = (\Psib_{\xb}^{RFF})^\ast \Kb [(\Psib_{\xb}^{RFF})^\ast ]^\dag,
\end{equation}
where $\Kb$ satisfies
\begin{equation}\label{eq:Koopman-solve2}
\Psib_{\xb}\Kb = \Psib_{\yb}.
\end{equation}
Multiplying~\eqref{eq:Koopman-solve2}
by $(\Psib_{\xb}^{RFF})^\ast$ and $[(\Psib_{\xb}^{RFF})^\ast]^\dag$ on the left and right respectively, and then using~\eqref{eq:innerprod}-\eqref{eq:KRFF}, leads to
\begin{equation*}
    (\Psib_{\xb}^{RFF})^\ast \Psib_{\xb}^{RFF} \Kb^{RFF} \approx (\Psib_{\xb}^{RFF})^\ast \Psib_{\yb}^{RFF},
\end{equation*}
which is the least squares normal equation for
\begin{equation}\label{eq:Koopman-solve3}
    \Psib_{\xb}^{RFF} \Kb^{RFF} =  \Psib_{\yb}^{RFF}.
\end{equation}
This directly connects the linear solves~\eqref{eq:Koopman-solve2} and~\eqref{eq:Koopman-solve3} for the Koopman matrix using kernel and random Fourier features, respectively. Moreover, from~\eqref{eq:innerprod}-\eqref{eq:KRFF},
\begin{align}\begin{split}\label{eq:FF-equiv}
    \Psib_{\xb}^{RFF} \Kb^{RFF} (\Psib_{\xb}^{RFF})^\dag \approx \Psib_{\xb} \Kb \Psib_{\xb}^\dag.
    \end{split}
\end{align}
In light of~\eqref{eq:inference}, equation~\eqref{eq:FF-equiv} shows that Fourier features and kernel features  
give (nearly) the same equation for inference.

Due to Proposition~\ref{propC} and the computations in~\eqref{eq:innerprod}-\eqref{eq:FF-equiv} above, 
random Fourier features~\eqref{eq:features-RFF} and kernel features~\eqref{eq:features} target essentially the same FKMD procedure whenever $R$ is sufficiently large. In practice, this means random Fourier features 
can be a more efficient way of solving the same problem.

\section*{Data availability}

 The Matlab code that runs FKMD in the experiments from Section~\ref{sec:ODE} and~\ref{sec:Lorenz} is available here\footnote{\url{https://github.com/davidaristoff/FKMD/tree/main}}. 

ERK activity reporters, cell line generation, and live-cell imaging have been described in detail in Davies et al~\cite{davies2020systems}. Here we utilize a dataset monitoring ERK activity in a tissue-like 3D extracellular matrix. Images were collected every 30 minutes over a 90-hour window. Single cells were segmented using Cellpose software~\cite{pachitariu2022cellpose} and tracked through time by matching cells to their closest counterpart at the previous time point. ERK reporter localization was monitored via the mean-centered and variance stabilized cross-correlation between the nuclear reporter and ERK activity reporter channels in the single-cell cytoplasmic mask.  
Single-cell trajectories up to 72 hours served as the training set to estimate the Koopman operator. Training and test set data is available and can be accessed via a Zenodo repository (https://doi.org/10.5281/zenodo.10849852).

\nocite{*}
\bibliography{aipsamp}

\end{document}